\documentclass[12pt]{amsart}
\usepackage{amssymb}
\usepackage{amsmath,latexsym}
\usepackage{amsfonts}
\usepackage{amsthm}
\usepackage{eucal}
\allowdisplaybreaks[4]
\newcommand{\R}{\mathbb{R}}

\newcommand{\vol}{{\rm vol}}
\newcommand{\Vol}{{\rm Vol}}
\newcommand{\supp}{{\rm supp}}

\numberwithin{equation}{section}

\newtheorem{thm}{Theorem}[section]
\newtheorem{defn}[thm]{Definition}
\newtheorem{lem}[thm]{Lemma}
\newtheorem{prop}[thm]{Proposition}

\newtheorem{remark}[thm]{Remark}

\begin{document}

\title[]{Perelman's $W$-functional on manifolds with conical singularities}
\author{Xianzhe Dai}
\address{
Department of Mathematics, East China Normal University, Shanghai, China, and
University of Californai, Santa Barbara
CA93106,
USA}
\email{dai@math.ucsb.edu}

\author{Changliang Wang}
\address{Department of Mathematics and Statistics, McMaster University, Hamilton, Ontario, Canada}
\email{wangc114@math.mcmaster.ca}
\date{}

\begin{abstract}
In this paper, we develop the theory of Perelman's $W$-functional on manifolds with isolated conical singularities. In particular, we show that the infimum of $W$-functional over a certain weighted Sobolev space on manifolds with isolated conical singularities is finite, and the minimizer exists, if the scalar curvature satisfies certain condition near the singularities. We also obtain an asymptotic order for the minimizer near the singularities.
\end{abstract}

\maketitle

\begin{section}{Introduction}
\noindent Let $(M,g)$ be a smooth compact Riemannian manifold without boundary. We recall some Riemannian functionals introduced by G. Perelman to study Ricci flows\cite{Per02}. The $\mathcal{F}$-functional is defined by
\begin{equation}\label{F-functional}
\mathcal{F}(g,f)=\int_{M}(R_{g}+|\nabla f|)e^{-f}d\vol_{g},
\end{equation}
where $R_{g}$ is the scalar curvature of the metric $g$, and $f$ is a smooth function on $M$. Let $u=e^{-\frac{f}{2}}$, then the $\mathcal{F}$-functional becomes
\begin{equation}\label{F-functional2}
\mathcal{F}(g,u)=\int_{M}(4|\nabla u|^{2}+R_{g}u^{2})d\vol_{g}.
\end{equation}
The Perelman's $\lambda$-functional is defined by
\begin{equation}\label{lambda-functional}
\lambda(g)=\inf\left\{\mathcal{F}(g,u) \mid \int_{M}u^{2}d\vol_{g}=1\right\}.
\end{equation}

Clearly, from $(\ref{lambda-functional})$ and $(\ref{F-functional2})$,  $\lambda(g)$ is the smallest eigenvalue of the Schr\"odinger operator $-4\Delta_{g}+R_{g}$. Starting from this point of view, we have extended Perelman's theory for the $\lambda$-functional to a class of singular manifolds, namely manifolds with isolated conical singularities in \cite{DW17}.

To take into account of scale, Perelman also introduces $W$-functional and $\mu$-functionial on smooth compact manifolds in \cite{Per02}. They play a crucial role in the study of singularities of Ricci flow. The $W$-functional is given by
\begin{equation}\label{W-functional}
\begin{aligned}
W(g, f, \tau)
&=\int_{M}[\tau(R_{g}+|\nabla f|^{2})+f-n]\frac{1}{(4\pi\tau)^{\frac{n}{2}}}e^{-f}d\vol_{g}\\
&=\frac{1}{(4\pi\tau)^{\frac{n}{2}}}\tau\mathcal{F}(g, f)+\frac{1}{(4\pi\tau)^{\frac{n}{2}}}\int_{M}(f-n)e^{-f} d\vol_{g},
\end{aligned}
\end{equation}
where $f$ is a smooth function, and $\tau>0$ is a scale parameter. As in $\mathcal{F}$-functional, let $u=e^{-\frac{f}{2}}$, then the $W$-functional becomes
\begin{equation}
W(g, u, \tau)=\frac{1}{(4\pi\tau)^{\frac{n}{2}}}\int_{M}[\tau(R_{g}u^{2}+4|\nabla u|^{2})-2u^{2}\ln u-nu^{2}]d\vol_{g}.
\end{equation}
The $\mu$-functional is defined by
\begin{equation}\label{mu-functional}
\mu(g, \tau)=\inf\left\{W(g, u, \tau) \mid u\in C^{\infty}(M),\ \  u>0,\ \  \text{and}\ \  \frac{1}{(4\pi\tau)^{\frac{n}{2}}}\int_{M}u^{2}d\vol_{g}=1\right\},
\end{equation}
for each $\tau>0$. It is well-know that for each fixed $\tau>0$ the existence of finite infimum follows from the Log Sobolev inequality on smooth compact Riemnnian manifolds, while the regularity of the minimizer follows from the elliptic estimates and Sobolev embedding. The nonlinear log term makes it a bit trickier than the eigenvalue problem (see, e.g. \S 11.3 in \cite{AH10} for details). For noncompact manifolds, the story  is different, and the $W$-functional on noncompact manifolds was studied in \cite{Zha12}.

In this paper we study the $W$-functional and $\mu$-functional on compact Riemannian manifolds with isolated conical singularities. Recall that, by a compact Riemannian manifold with isolated conical singularities we mean a singular manifold $(M, g, S)$ whose singular set $S$ consists of finite many points and its regular part $(M\setminus S, g)$ is a smooth Riemannian manifold. Moreover, near the singularities, the metric is asymptotic to a (finite) metric cone $C_{(0, 1]}(N)$ where $N$ is a compact smooth Riemannian manifold with metric $h_0$ which will be called a cross section (see \S 1 for the precise definition). Our main result is the following theorem.

\begin{thm}
Let $(M^{n}, g, S)$ $(n\geq3)$ be a compact Riemannian manifold with isolated conical singularities. If the scalar curvature of the cross section at the conical singularity $R_{h_{0}}>(n-2)$ on $N$, then for each fixed $\tau>0$,
\begin{equation}\label{FiniteInfimum}
\inf\left\{W(g, u, \tau)\mid u\in H^{1}(M), u>0, \left\|\frac{1}{(4\pi\tau)^{\frac{n}{4}}}u\right\|_{L^{2}(M)}=1\right\}>-\infty
\end{equation}
Here $H^{1}(M)$ is the weighted Sobolev space defined in $(\ref{WSN1})$.

Moreover, there exists $u_{0}\in C^{\infty}(M\setminus S)$ that realizes the infimum in $(\ref{FiniteInfimum})$. Furthermore, if $(M^{n}, g, S)$ satisfies the asymptotic condition $AC_{1}$ defined in $(\ref{AC})$, then near each singularity, the minimizer satisfies
\begin{equation}\label{AsymptoticBehavior}
u_{0}=o(r^{-\alpha}), \ \ \text{as}\ \ r\rightarrow0,
\end{equation}
for any $\alpha>\frac{n}{2}-1$. Here $r$ is the radial variable on each conical neighborhood of the singularities, and $r=0$ corresponds to the singular points.
\end{thm}

\begin{remark}
In a recent paper \cite{Ozu17}, T. Ozuch studied Perelman's functionals on cones and showed that the infimum in $(\ref{FiniteInfimum})$ is finite.
\end{remark}

In \cite{DW17}, we have shown that the infimum of the $\mathcal{F}$-functional over the weighted Sobolev space $H^{1}(M^{n})$ is finite if $R_{h_{0}}>(n-2)$ on the cross section. In order to control the term involving $\ln u$ in the $W$-functional, similar as in the smooth compact case, one  uses Log Sobolev inequality on compact manifolds with isolated conical singularities, which follows from a $L^{2}$ Sobolev inequality on compact manifolds with isolated conical singularities. Then we conclude that the infimum in $(\ref{FiniteInfimum})$ is finite. The $L^{2}$ Sobolev inequality on compact manifolds with isolated conical singularities is established in \cite{DY}. Clearly, it suffices to establish the inequality on a  metric cone. For this, a Hardy inequality on model cones, which follows from the classical weighted Hardy inequality, will play an important role.

Then we use the direct method in the calculus of variations to show the existence of a minimizer of the $W$-functional. Basically, we follow the strategy in the smooth compact case given detailed in \cite{AH10}. However, there are some obvious difference between smooth compact case and the singular case and there are more difficulties in the singular case that we need to deal with. For example, the scalar curvature, which appears in the $W$-functional, goes to infinity near the singularities. Thus in order to deal with the limit for the term involving $\ln u$ in the $W$-functional, instead of using the compactness of classical Sobolev embedding, we need to use the compactness of certain weighted Sobolev embedding obtained in Proposition $\ref{CompactWeightedSobolevEmbedding}$ below. Then the regularity of the minimizer follows from the classical elliptic equation theory, since this is a local problem.

Finally, we use certain weighted Sobolev embedding and weighted elliptic estimates to obtain the asymptotic behavior $(\ref{AsymptoticBehavior})$ for the minimizer. These weighted Sobolev embedding and weighted elliptic estimates follow from classical Sobolev embedding, interior elliptic estimates, and an useful scaling technique. The scaling technique can be applied in this problem because of the obvious homogeneity of a model cone along the radial direction. And the scaling technique has been demonstrated to be very useful in studying weighted norms and weighted spaces on non-compact manifolds. For a brief survey about its applications, we refer to \S1 in \cite{Bar86}.
\end{section}


\begin{section}{Manifolds with Isolated Conical Singularities}
\noindent As mentioned in the introduction, roughly speaking, a compact Riemannian manifold with isolated conical singularities is a singular manifold $(M, g)$ whose singular set $S$ consists of finite many points and its regular part $(M\setminus S, g)$ is a smooth Riemannian manifold. Moreover, near the singularities, the metric is asymptotic to a (finite) metric cone $C_{(0, 1]}(N)$ where $N$ is a compact smooth Riemannian manifold with metric $h_0$. More precisely,

\noindent  \begin{defn}\label{ManifoldWithConicalSingularities}
We say $(M^{n},d,g,x_{1},\cdots,x_{k})$ is a compact Riemannian manifold with isolated conical singularities at $x_{1},\cdots,x_{k}$, if
\begin{enumerate}
\item $(M,d)$ is a compact metric space,
\item $(M_{0},g|_{M_{0}})$ is an n-dimensional smooth Riemannian manifold, and the Riemannian metric $g$ induces the given metric $d$ on $M_{0}$, where $M_{0}=M\setminus\{x_{1},\cdots,x_{k}\}$,
\item for each singularity $x_{i}$, $1\leq i\leq k$, their exists a neighborhood $U_{x_{i}}\subset M$ of $x_{i}$ such that $U_{x_{i}}\cap \{x_{1},\cdots, x_{k}\}=\{x_{i}\}$, $(U_{x_{i}}\setminus\{x_{i}\},g|_{U_{x_{i}}\setminus\{x_{i}\}})$ is isometric to $((0,\varepsilon_{i})\times N_{i},dr^{2}+r^{2}h_{r})$ for some $\varepsilon_{i}>0$ and a compact smooth manifold $N_{i}$, where $r$ is a coordinate on $(0, \varepsilon_{i})$ and $h_{r}$ is a smooth family of Riemannian metrics on $N_{i}$ satisfying $h_{r}=h_{0}+o(r^{\alpha_{i}})$ as $r\rightarrow 0$, where $\alpha_{i}>0$ and $h_{0}$ is a smooth Riemannian metric on $N_{i}$.
\end{enumerate}
Moreover, we say a singularity $p$ is a cone-like singularity, if the metric $g$ on a neighborhood of $p$ is isometric to $dr^{2}+r^{2}h_{0}$ for some fixed metric $h_{0}$ on the cross section $N$.
\end{defn}

In our case, as usual, one does analysis away from the singular set. And in the above definition, we only require the zeroth order asymptotic condition $h_{r}=h_{0}+o(r^{\alpha})$, as $r\rightarrow 0$, for the family of metrics $h_{r}$ on the cross section $N$ with parameter $r>0$. However, in some problems we need certain higher order asymptotic conditions for $h_{r}$ as follows. We say that a compact Riemannian manifold $(M^{n}, g, x)$ with a single conical singularity at $x$ satisfies the condition $AC_{k}$, if
\begin{equation}\label{AC}
r^{i-1}|\nabla^{i}(h_{r}-h_{0})|\leq C_{i}<+\infty,
\end{equation}
for some constant $C_{i}$, and each $1\leq i\leq k$, near $x$.
\end{section}

\begin{remark}
For simplicity, in the rest of this paper, we will only work on manifolds with a single conical point as there is no essential difference between the case of a single singular point  and that of multiple isolated singularities. All our work and results for manifolds with a single conical point go through for manifolds with isolated conical singularities.
\end{remark}

For the simplicity of notations, we will use $(M^{n}, g, x)$ to denote a compact Riemannian manifold with a single conical singularity at $x$, because the metric $d$ is determined by the Riemannian metric $g$.


\begin{section}{Sobolev and weighted Sobolev embedding}

\noindent In this section, we recall certain weighted Sobolev spaces on compact Riemannian manifolds with conical singularities. And we establish the identification of some of them with the usual (unweighted) Sobolev spaces. Moreover, we also review and establish some Sobolev and weighted Sobolev embedding on compact Riemannian manifolds with isolated conical singularities.

Various weighted Sobolev spaces and their properties have been introduced and intensively studied in different settings, e.g. on complete non-compact manifolds with certain asymptotic behavior at infinity (see, e.g. \cite{Bar86}, \cite{Can81}, \cite{CBC81}, \cite{Loc81}, \cite{LP87}, \cite{LM85}, \cite{McO79}, \cite{NW73}, and \cite{Wan17}), or on various interesting bounded domains in $\mathbb{R}^{n}$ (see, e.g. \cite{KMR97}, \cite{Kuf85}, \cite{Tri78}, and \cite{Tur00}).
The ones most closely related to our setting are weighted Sobolev spaces introduced in \cite{BP03} on compact Riemannian manifolds with isolated tame conical singularities. Inspired by these pioneering work, we introduce appropriately weighted Sobolev spaces on compact Riemannian manifolds with isolated conical singularities as follows.

Let $(M^{n},g,x)$ be a compact Riemannian manifold with a single conical singularity at $x$, and $U_{x}$ be a conical neighborhood of $x$ such that $(U_{x}\setminus \{x\},g|_{U_{x}\setminus \{x\}})$ is isometric to $((0,\epsilon_{0})\times N, dr^{2}+r^{2}h_{r})$. For each $k\in\mathbb{N}$, $p\geq1$, and $\delta\in\mathbb{R}$, we define the weighted Sobolev space $W^{k, p}_{\delta}(M)$ to be the completion of $C^{\infty}_{0}(M\setminus\{x\})$ with respect to the weighted Sobolev norm
\begin{equation}\label{WSN}
\|u\|_{W^{k, p}_{\delta}(M)}=\left(\int_{M}(\sum^{k}_{i=0}\chi^{p(\delta-i)+n}|\nabla^{i}u|^{p})d\vol_{g}\right)^{\frac{1}{p}},
\end{equation}
where $\nabla^{i}u$ denotes the $i$-times covariant derivative of the function $u$, and $\chi\in C^{\infty}(M\setminus\{x\})$ is a positive weight function satisfying
\begin{equation}\label{WeightFunction}
\chi(y)=
\begin{cases} 1 & \text{if} \ \ y\in M\setminus U_{x}, \\ \frac{1}{r} & \text{if} \ \  y=(r, \theta)\in U_{x}\subset M, \ \  \text{and} \ \ r<\frac{\epsilon_{0}}{4},
\end{cases}
\end{equation}
and $0<(\chi(y))^{-1}\leq 1$ for all $y\in M\setminus\{x\}$.

For the simplicity of notations, as in \cite{DW17}, we set $H^{k}(M)\equiv W^{k, 2}_{k-\frac{n}{2}}(M)$, and
\begin{equation}\label{WSN1}
\|u\|^{2}_{H^{k}(M)}\equiv\int_{M}(\sum^{k}_{i=0}\chi^{2(k-i)}|\nabla^{i}u|^{2})d\vol_{g}.
\end{equation}

As usual, $W^{k, p}(M)$ denotes the completion of $C^{\infty}_{0}(M\setminus\{x\})$ with respect to the usual Sobolev norm
\begin{equation}\label{SobolevNorm}
\|u\|_{W^{k, p}(M)}=\left(\int_{M}(\sum^{k}_{i=0}|\nabla^{i}u|^{p})d\vol_{g}\right)^{\frac{1}{p}}.
\end{equation}

We now recall a weighted Hardy inequality {\rm (}see, e.g. {\bf 330} on p. 245 in \cite{HLP34}{\rm )}, and from which derive a Hardy inequality on metric cones. Later we will see that the Hardy inequality on cone will play an important role for establishing Sobolev embedding on manifolds with isolated conical singularities.

For $p>1$ and $a\neq1$, we have
\begin{equation}\label{WeightedHardyInequality}
\int^{\infty}_{0}|f|^{p}x^{-a}dx\leq\left(\frac{p}{|a-1|}\right)^{p}\int^{\infty}_{0}|f^{\prime}(x)|^{p}x^{p-a}dx,
\end{equation}
for any $f\in C^{\infty}_{0}((0, \infty))$.

This weighted Hardy inequality implies a Hardy inequality on an $n$-dimensional metric cone $(C(N)=(0, \infty)\times N^{n-1}, g=dr^{2}+r^{2}h)$ over a smooth compact Riemannian manifold $(N^{n-1}, h)$. Indeed, for $p>1$ and $k\in\mathbb{N}$ with $pk\neq n$, and any $u\in C^{\infty}_{0}(C(N))$,
\begin{equation}\label{ConeHardyInequality}
\begin{aligned}
\int_{C(N)}\frac{|u|^{p}}{r^{pk}}d\vol_{g}
&=\int_{N}\int^{\infty}_{0}\frac{|u|^{p}(r, \theta)}{r^{pk}}r^{n-1}drd\vol_{h}\\
&=\int_{N}\int^{\infty}_{0}|u|^{p}(r, \theta)r^{n-1-pk}drd\vol_{h}\\
&\leq \left(\frac{p}{|n-pk|}\right)^{p}\int_{N}\int^{\infty}_{0}\left|\frac{\partial u}{\partial r}\right|^{p}(r, \theta)r^{n-1-p(k-1)}drd\vol_{h}\\
&\leq \left(\frac{p}{|n-pk|}\right)^{p}\int_{C(N)}\frac{|\nabla u|_{g}^{p}}{r^{p(k-1)}}d\vol_{g}.
\end{aligned}
\end{equation}
Here, for the first inequality, we used the inequality  $(\ref{WeightedHardyInequality})$ for each $u(r, \theta)$ with fixed $\theta$ and $a=pk+1-n$. And the last inequality follows from
$|\nabla u|_{g}=\left(\left|\frac{\partial u}{\partial r}\right|^{2}+\frac{1}{r^{2}}|\nabla_{N}u|_{h}^{2}\right)^{\frac{1}{2}}$, where $\nabla_{N}$ is the covariant derivative on $N$ with respect to the metric $h$.

Then combining with the Kato's inequality, $|\nabla|\nabla^{k}u||\leq|\nabla^{k+1}u|$ for any smooth function $u$ and non-negative integer $k$, this Hardy inequality on metric cones directly implies the following equivalence between the weighted Sobolev norms and the usual Sobolev norms.
\begin{lem}
Let $(M^{n}, g, x)$ be a compact Riemannian manifold with a single conical singularity at $x$. For each $p>1$ and $k\in\mathbb{N}$ with $p\, i\neq n$ for all $i=1, 2, \cdots, k$, if $(M^{n}, g, x)$ satisfies the condition $AC_{k-1}$ near $x$ defined in $(\ref{AC})$, then we have for any $u\in C^{\infty}_{0}(M\setminus\{x\})$,
\begin{equation}
\|u\|_{W^{k, p}(M)}\leq\|u\|_{W^{k, p}_{k-\frac{n}{p}}(M)}\leq C(g, n, p, k)\|u\|_{W^{k, p}(M)},
\end{equation}
for a constant $C(g, n, p, k)$ depending on $g, n, p$, and $k$.

Consequently, we have $W^{k, p}_{k-\frac{n}{p}}(M^{n})=W^{k, p}(M^{n})$ for each $p>1$ and $k\in\mathbb{N}$ with $p\, i\neq n$ for all $i=1, 2, \cdots, k$.
\end{lem}

Even though we have obtained that some weighted Sobolev norms are equivalent to the usual Sobolev norms, sometimes it is still more convenient to use weighted Sobolev norms. For example, a certain homogeneity of weighted Sobolev norms on metric cones has been demonstrated to be very useful in \S 8 in \cite{DW17} and the proof of Proposition $\ref{WeightedSobolevEmbedding}$ below. Moreover, we only have equivalence between the usual Sobolev norms and weighted Sobolev norms for special weight indices $\delta=k-\frac{n}{p}$ with $k, p$, and $n$ satisfying certain conditions. However, in some problems, we have to use weighted Sobolev norms with more general weight indices, e.g. in \S4 and \S5.

Another application of the Hardy inequality obtained in $(\ref{ConeHardyInequality})$ is the following Sobolev inequality on the metric cone $(C(N)=(0, \infty)\times N^{n-1}, g=dr^{2}+r^{2}h)$.
\begin{lem}\label{ConeSobolevInequality}
For $1<p<n$, and any $u\in C^{\infty}_{0}(C(N))$, we have
\begin{equation}
\|u\|_{L^{q}(C(N))}\leq C\|\nabla u\|_{L^{p}(C(N))},
\end{equation}
for a constant $C$ only depending on the cross section $(N^{n-1}, h)$ and $p$, where $q=\frac{np}{n-p}$.
\end{lem}

{\em Sketch of the proof of Lemma $\ref{ConeSobolevInequality}$}:
The Sobolev inequality in Lemma $\ref{ConeSobolevInequality}$ has been established in \cite{DY} for the case $p=2$. And no essential difference between $p=2$ case and general case in Lemma $\ref{ConeSobolevInequality}$. For proof we refer to \cite{DY}. The basic idea is to choose a finite sufficiently small open cover for the cross section $(N^{n-1}, h)$ so that each piece can be embedded into Euclidean unit sphere $\mathbb{S}^{n-1}$ and the metric $h$ restricted onto each small piece is equivalent to standard metric on the Euclidean unit sphere. Then on the cone over each small piece the metric $h$ is equivalent to standard Euclidean metric on $\mathbb{R}^{n}$. Then we choose a partition of unity $\{\rho_{i}\}^{N}_{i=1}$ subject to the open cover chose for $(N^{n-1}, h)$. If we let $\pi: C(N)\rightarrow N$ be the natural projection. Then an important observation pointed in \cite{DY} is the pointwise estimate:
\begin{equation}\label{PartitionOfUnityEstimate}
|\nabla(\pi^{*}\rho_{i})|(r, \theta)\leq C_{i}r^{-1}
\end{equation}
where $C_{i}$ is a constant, and $\nabla$ is the covariant derivative with respect to $g=dr^{2}+r^{2}h$ on the cone. Then combining $(\ref{ConeHardyInequality})$ and $(\ref{PartitionOfUnityEstimate})$, one can easily obtain Sobolev inequality in Lemma $\ref{ConeSobolevInequality}$.

By applying the Kato's inequality again, Lemma $\ref{ConeSobolevInequality}$ implies the following Sobolev inequalities and Sobolev embedding on compact Riemannian manifolds with isolated conical singularities.
\begin{prop}\label{SobolevInequality1}
Let $(M^{n}, g, x)$ be a compact Riemannian manifold with a single conical singularity at $x$. For each $1<p<n$, we have
\begin{enumerate}
\item for any $u\in C^{\infty}_{0}(M\setminus\{x\})$
\begin{equation}
 \|u\|_{W^{l, q}(M)}\leq C(M, g, p, k)\|u\|_{W^{k, p}(M)},
\end{equation}
for any $1\leq q\leq q_{l}$, where $C(M, g, p, k)$ is a constant, and $l<k$ and $q_{l}$ satisfy $\frac{1}{q_{l}}=\frac{1}{p}-\frac{k-l}{n}>0$,
\item hence continuous embedding $W^{k, p}(M)\subset W^{l, q}(M)$, for any $1\leq q\leq q_{l}$,
\end{enumerate}
\end{prop}

Thus, the Sobolev embedding on compact manifolds with isolated conical singularities relies on the weighted $L^{p}$-Hardy inequality $(\ref{WeightedHardyInequality})$ for $p>1$, which is known not to be true in the case of $p=1$. So in general, we do not have Sobolev embeddings on manifolds with isolated conical singularities in the case of $p=1$. However, in the following, we will see that one always has weighted Sobolev embeddings for all $p\geq1$ on compact manifolds with isolated conical singularities. The key idea is to use a homogeneity of weighted Sobolev norms on metric cones.

\begin{prop}\label{WeightedSobolevEmbedding}
Let $(M^{n}, g, x)$ be a compact Riemannian manifold with a single conical singularity at $x$ satisfying the condition $AC_{k-1}$ defined in $(\ref{AC})$. For each $1\leq p<n$, $\delta\in\mathbb{R}$, we have for any $u\in C^{\infty}_{0}(M\setminus\{x\})$
\begin{equation}\label{WeightedSobolevInequality}
\|u\|_{W^{l, q}_{\delta}(M)}\leq C(g, n, p, k, l)\|u\|_{W^{k, p}_{\delta}(M)},
\end{equation}
for any $1\leq q\leq q_{l}$, where $C(M, g, p, k)$ is a constant, and $l<k$ and $q_{l}$ satisfy $\frac{1}{q_{l}}=\frac{1}{p}-\frac{k-l}{n}>0$. And therefore, we have continuous embeddings $W^{k, p}_{\delta}(M)\subset W^{l, q}_{\delta}(M)$, for $l<k$ and $q\leq q_{l}$.
\end{prop}
\begin{proof}
Clearly, it suffices to show the weighted Sobolev inequality $(\ref{WeightedSobolevInequality})$ on a finite metric cone $(C_{\epsilon}(N)=(0, \epsilon)\times N, g=dr^{2}+r^{2}h_{0})$ over the smooth compact Riemannian manifold $(N^{n-1}, h_{0})$.

Similarly as in the proof of Lemma 8.1 in \cite{DW17}, let $u(r, \theta)\in C^{\infty}_{0}(C_{\epsilon}(N))$, where $\theta$ is a local coordinate on $N$, and set
\begin{equation}
u_{a}(r, \theta)=u(ar, \theta),
\end{equation}
for a positive constant $a$. Let $C_{r_{1}, r_{2}}=(r_{1}, r_{2})\times N$ be an annulus on the finite metric cone $C_{\epsilon}(N)$, for $0\leq r_{1}<r_{2}\leq\epsilon$. Then by a simple change of variable, we can see
\begin{equation}\label{SobolevHomogeneity}
\|u\|_{W^{k, p}_{\delta}(C_{ar_{1}, ar_{2}})}=a^{-\delta}\|u\|_{W^{k, p}_{\delta}(C_{r_{1}, r_{2}})}.
\end{equation}
Thus,
\begin{equation}
\begin{aligned}
\|u\|_{W^{l, q}_{\delta}(C_{\epsilon})}
&=\sum^{\infty}_{j=0}\|u\|_{W^{l, q}_{\delta}(C_{(\frac{1}{2})^{j+1}\epsilon, (\frac{1}{2})^{j}\epsilon})}\\
&=\sum^{\infty}_{j=0}(\frac{1}{2})^{-j\delta}\|u_{(\frac{1}{2})^{j}}\|_{W^{l, q}_{\delta}(C_{\frac{1}{2}\epsilon, \epsilon})}\\
&\leq C(\epsilon)\sum^{\infty}_{j=0}(\frac{1}{2})^{-j\delta}\|u_{(\frac{1}{2})^{j}}\|_{W^{k, p}_{\delta}(C_{\frac{1}{2}\epsilon, \epsilon})}\\
&=C(\epsilon)\sum^{\infty}_{j=0}\|u\|_{W^{k, p}_{\delta}(C_{(\frac{1}{2})^{j+1}\epsilon, (\frac{1}{2})^{j}\epsilon})}\\
&=C(\epsilon)\|u\|_{W^{l, q}_{\delta}(C_{\epsilon})},
\end{aligned}
\end{equation}
where $C(\epsilon)$ is a constant depending on $\epsilon$ but not depending on $j$.
Here, for the above inequality, we used the usual Sobolev inequality on the compact manifold $\overline{C_{\frac{1}{2}\epsilon, \epsilon}}=([\frac{1}{2}\epsilon, \epsilon]\times N, dr^{2}+r^{2}h_{0})$ with boundary.
\end{proof}

\begin{remark}
In Theorem 3.3 in \cite{BP03}, on a compact Riemannian manifold with isolated tame conical singularities $M^{n}$ of dimension $n$, the continuous embedding $W^{k, 2}_{1-\frac{n}{2}}(M^{n})\subset L^{q}(M^{n})$ for $2\leq q\leq\frac{2n}{n-2}$ with $n\geq5$ has been shown, and these can be considered as special cases of Proposition $\ref{WeightedSobolevEmbedding}$, since $\|u\|_{L^{q}(M)}\leq\|u\|_{W^{0, q}_{1-\frac{n}{2}}(M)}$ for all $2\leq q\leq\frac{2n}{n-2}$ with $n\geq5$.
\end{remark}

Finally, we show the following compactness property for a weighted Sobolev embedding obtained in Proposition $\ref{WeightedSobolevEmbedding}$. This compactness property will be used in showing the existence of the minimizer of the $W$-functional.

\begin{prop}\label{CompactWeightedSobolevEmbedding}
Let $(M^{n}, g, x)$ be a compact Riemannian manifold with a single conical singularity at $x$. The embedding $W^{1, 1}_{1-n}(M)\subset L^{q}(M)$ is compact for any $1\leq q<\frac{n}{n-1}$.
\end{prop}
\begin{proof}
The embedding follows from $W^{1, 1}_{1-n}(M)\subset W^{0, q}_{1-n}\subset L^{q}(M)$ for $1\leq q<\frac{n}{n-1}$. The first inclusion is given in Proposition $\ref{WeightedSobolevEmbedding}$. And the second inclusion follows from $q(1-n)+n>0$ and the definition of weighted Sobolev norms $(\ref{WSN})$. Thus, we only need to show the compactness of the embedding. For that we will use the idea of the proof of Lemma $\ref{ConeSobolevInequality}$ described right after the lemma.

Choose $0<\epsilon<\frac{\epsilon_{0}}{10}$ sufficiently small so that $C_{3\epsilon}(N)=(0, 3\epsilon)\times N\subset M$ is a conical neighborhood of $x$, and
\begin{equation}
\frac{1}{2}(g_{0}=dr^{2}+r^{2}h_{0})|_{C_{2\epsilon}(N)}\leq (g=dr^{2}+r^{2}h_{r})|_{C_{2\epsilon}(N)}\leq 2(g_{0}=dr^{2}+r^{2}h_{0})|_{C_{2\epsilon}(N)}.
\end{equation}
Then choose a smooth function $\phi_{1}$ on $M\setminus\{x\}$ with $\phi_{1}\equiv1$ on $C_{\epsilon}(N)\subset M\setminus\{x\}$, $\supp(\phi_{1})\subset C_{2\epsilon}(N)$, $0\leq\phi_{1}\leq1$, and $\phi_{1}|_{C_{2\epsilon}(N)}=\phi_{1}(r, \theta)$ is a radial function, i.e. only depends on $r$. And set $\phi_{2}=1-\phi_{1}$ on $M\setminus\{x\}$.

Let $\{u_{m}\}^{\infty}_{m=1}\subset W^{1, 1}_{1-n}(M)$ be a bounded sequence, i.e.
\begin{equation}
\|u_{m}\|_{W^{1, 1}_{1-n}(M)}=\int_{M}(|\nabla u_{m}|+\chi|u_{m}|)d\vol_{g}\leq A,
\end{equation}
for some uniform constant $A$, where $\chi$ is the weight function given in $(\ref{WeightFunction})$.

We choose a finite sufficiently small open cover $\{U_{i}\}^{i_{0}}_{i=1}$ of $N^{n-1}$, such that $U_{i}$ can be embedded into the Euclidean unit sphere $\mathbb{S}^{n-1}$, and
\begin{equation}
\frac{1}{2}g_{\mathbb{S}^{n-1}}\leq h_{0}|_{U_{i}}\leq 2g_{\mathbb{S}^{n-1}},
\end{equation}
for all $1\leq i\leq i_{0}$. Consequently, $C_{2\epsilon}(U_{i})=(0, 2\epsilon)\times N$ can be embedded into $\mathbb{R}^{n}$ as $\Phi_{i}: C_{2\epsilon}(U_{i})\rightarrow B_{1}(0)\subset\mathbb{R}^{n}$, and
\begin{equation}
\frac{1}{4}\Phi^{*}_{i}(g_{\mathbb{R}^{n}})\leq(g=dr^{2}+r^{2}h_{r})|_{C_{2\epsilon}(U_{i})}\leq4\Phi^{*}_{i}(g_{\mathbb{R}^{n}}),
\end{equation}
for all $1\leq i\leq i_{0}$, where $B_{1}(0)$ is the unit ball centered at the origin in $\mathbb{R}^{n}$.

We also choose a partition of unity $\{\rho_{i}\}^{i_{0}}_{i=1}$ subject to the open cover $\{U_{i}\}^{i_{0}}_{i=1}$ of $N^{n-1}$. Then for each $1\leq i\leq i_{0}$, and $m\in\mathbb{N}$, $(\pi^{*}(\rho_{i})\cdot\phi_{1}\cdot u_{m})\circ\Phi^{-1}_{i}\in C^{\infty}_{0}(\overline{B_{1}(0)})$, and
\begin{align*}
&\ \ \ \ \ \|(\pi^{*}(\rho_{i})\cdot\phi_{1}\cdot u_{m})\circ\Phi^{-1}_{i}\|_{W^{1, 1}(\overline{B_{1}(0)})}\\
&=\int_{\overline{B_{1}(0)}}\left(|\nabla((\pi^{*}(\rho_{i})\cdot\phi_{1}\cdot u_{m})\circ\Phi^{-1}_{i})|_{g_{\mathbb{R}^{n}}}+|(\pi^{*}(\rho_{i})\cdot\phi_{1}\cdot u_{m})\circ\Phi^{-1}_{i}|\right)d\vol_{g_{\mathbb{R}^{n}}}\\
&\leq\int_{\overline{B_{1}(0)}}[\left(4|\nabla(\pi^{*}(\rho_{i}))|_{g}|\phi_{1}\cdot u_{m}|\right)\circ\Phi^{-1}_{i}+\left(\pi^{*}(\rho_{i})4|\nabla(\phi_{1}\cdot u_{m})|_{g}\right)\circ\Phi^{-1}_{i}]d\vol_{g_{\mathbb{R}^{n}}}\\
&\ \ \ \ +\int_{\overline{B_{1}(0)}}|(\pi^{*}(\rho_{i})\cdot\phi_{1}\cdot u_{m})\circ\Phi^{-1}_{i}|d\vol_{g_{\mathbb{R}^{n}}}\\
&\leq4^{n+1}C\int_{C_{2\epsilon}(N)}\left(\frac{1}{r}|u_{m}|+|\nabla u_{m}|_{g}+|u_{m}|\right)d\vol_{g}\\
&\leq4^{n+1}C\int_{M}(|\nabla u|_{g}+\chi|u_{m}|)d\vol_{g}\\
&=4^{n+1}C\|u_{m}\|_{W^{1, 1}_{1-n}}\leq4^{n+1}C\cdot A,
\end{align*}
where $C$ and $A$ are constants independent of $m$ and $i$.

Then we choose a finite open cover $\{V_{j}\}^{j_{0}}_{j=1}$ for the compact manifold $M\setminus C_{\epsilon}(N)$ with smooth boundary $(N, \epsilon^{2}h_{\epsilon})$ such that the metric $g$ on $M$ restricted on each $V_{j}$ is quasi-isometric to the standard n-dimensional unit ball or a subset of the unit ball, say $\Psi_{j}: V_{j}\rightarrow B_{1}(0)\subset\mathbb{R}^{n}$. We also choose a partition of unity $\{\psi_{j}\}^{j_{0}}_{j=1}$ subject to the open cover. Then for each $1\leq j\leq j_{0}$, and $m\in\mathbb{N}$, $(\psi_{j}\cdot\phi_{2}\cdot u_{m})\circ\Psi^{-1}_{j}\in C^{\infty}_{0}(\overline{B_{1}(0)})$, and
\begin{equation}
\|(\psi_{j}\cdot\phi_{2}\cdot u_{m})\circ\Psi^{-1}_{j}\|_{W^{1, 1}(\overline{B_{1}(0)})}\leq C^{\prime}\|u_{m}\|_{W^{1, 1}(M)}\leq C^{\prime}\|u_{m}\|_{W^{1, 1}_{1-n}(M)}\leq C^{\prime}\cdot A,
\end{equation}
for constants $C^{\prime}$ and $A$ independent of $m$ and $i$.

Then for each fixed $1\leq q<\frac{n}{n-1}$, by the compactness of usual Sobolev embedding on the closed unit ball in $\mathbb{R}^{n}$, we can choose a subsequence of $\{u_{m}\}^{\infty}_{i=1}$, which is still denoted by $\{u_{m}\}$, such that $\{\pi^{*}(\rho_{1})\cdot\phi_{1}\cdot u_{m}\}^{\infty}_{m=1}$ is a Cauchy sequence in $L^{q}(M)$. And do this for $i=2, \cdots, i_{0}$, and then $j=1, \cdots, j_{0}$, and the subsequences from each step. Finally, we can obtain a subsequence of the original sequence $\{u_{m}\}$, which is still denoted by $\{u_{m}\}$, such that all $\{\pi^{*}(\rho_{i})\cdot\phi_{1}\cdot u_{m}\}$ for $1\leq i\leq i_{0}$ and all $\{\psi_{j}\cdot\phi_{2}\cdot u_{m}\}$ for $1\leq j\leq j_{0}$ are Cauchy sequences in $L^{q}(M)$. Therefore, $\{u_{m}\}$ is a Cauchy sequence in $L^{q}(M)$, since
\begin{equation*}
\begin{aligned}
\|u_{m}-u_{m^{\prime}}\|_{L^{q}(M)}
&\leq\sum^{i_{0}}_{i=1}\|\pi^{*}(\rho_{i})\cdot\phi_{1}\cdot u_{m}-\pi^{*}(\rho_{i})\cdot\phi_{1}\cdot u_{m^{\prime}}\|_{L^{q}(M)}\\
&\ \ \ \ +\sum^{j_{0}}_{j=1}\|\psi_{j}\cdot\phi_{2}\cdot u_{m}-\psi_{j}\cdot\phi_{2}\cdot u_{m^{\prime}}\|_{L^{q}(M)}.
\end{aligned}
\end{equation*}
This completes the proof.
\end{proof}
\end{section}


\begin{section}{Finite lower bound of $W$-functional}

\noindent In this section, we show that on a manifold with a single conical singularity $(M^{n}, g, x)$ the $W$-functional has a finite lower bound over all functions in $H^{1}(M)$. By the work in \cite{DW17} about the $\lambda$-functional on these manifolds, the key here is to obtain a bound for the term $\int_{M}u^{2}\log ud\vol_{g}$ in the definition of the $W$-functional.

By using the $L^{2}$ Sobolev inequality on compact manifolds with isolated conical singularities obtained in Proposition $\ref{SobolevInequality1}$, in particular, $k=1, p=2$ case, it is well-known that we can derive the following Logarithmic Sobolev inequality (see, e.g. Lemma 5.8 in \cite{CLN06}).

\begin{lem}
Let $(M^{n}, g, x)$ be a compact Riemannian manifold with a single conical singularity at $x$. For any $a>0$, there exists a constant $C(a, g)$ such that if $u\in W^{1, 2}(M)$ with $u>0$ and $\|u\|_{L^{2}(M)}=1$, then
\begin{equation}
\int_{M}u^{2}\ln ud\vol_{g}\leq a\int_{M}|\nabla u|d\vol_{g}+C(a, g).
\end{equation}
\end{lem}

Then for any $a>0$, and $u\in H^{1}(M)\equiv W^{1, 2}_{1-\frac{n}{2}}(M)\subset W^{1, 2}(M)$ with $u>0$ and $\left\|\frac{1}{(4\pi\tau)^{\frac{n}{4}}}u\right\|_{L^{2}(M)}=1$, we have
\begin{equation}\label{W-functionalLowerBound}
\begin{aligned}
W(g, u, \tau)
&=\frac{1}{(4\pi\tau)^{\frac{n}{2}}}\int_{M}[\tau(R_{g}u^{2}+4|\nabla u|^{2})-2u^{2}\ln u-nu^{2}]d\vol_{g}\\
&\geq \frac{1}{(4\pi\tau)^{\frac{n}{2}}}\int_{M}\tau(R_{g}u^{2}+4|\nabla u|^{2})d\vol_{g}-a\frac{1}{(4\pi\tau)^{\frac{n}{2}}}\int_{M}|\nabla u|^{2}d\vol_{g}\\
&\ \ \ \ -\frac{n}{2}\ln(4\pi\tau)-C(a, g)-n\\
&=\frac{\tau}{(4\pi\tau)^{\frac{n}{2}}}\int_{M}(R_{g}u^{2}+(4-\frac{a}{\tau})|\nabla u|^{2})d\vol_{g}-\frac{n}{2}\ln(4\pi\tau)-C(a, g)-n.
\end{aligned}
\end{equation}

Moreover, for each fixed $\tau>0$, by Remark 1.3 in \cite{DW17}, we can choose a sufficiently small $a>0$ such that
\begin{equation*}
\inf\left\{\int_{M}(R_{g}u^{2}+(4-\frac{a}{\tau})|\nabla u|^{2})d\vol_{g}\mid u\in H^{1}(M), u>0, \left\|\frac{1}{(4\pi\tau)^{\frac{n}{4}}}u\right\|_{L^{2}(M)}=1\right\}>-\infty,
\end{equation*}
if $R_{h_{0}}>(n-2)$ on the cross section of at the conical singularity.

Thus, we have
\begin{thm}\label{ExistenceOfMinimizer}
Let $(M^{n}, g, x)$ be a compact Riemannian manifold with a single conical singularity at $x$. If the scalar curvature of the cross section at the conical singularity $R_{h_{0}}>(n-2)$ on $N$, then for each fixed $\tau>0$,
\begin{equation}
\inf\left\{W(g, u, \tau)\mid u\in H^{1}(M), u>0, \left\|\frac{1}{(4\pi\tau)^{\frac{n}{4}}}u\right\|_{L^{2}(M)}=1\right\}>-\infty
\end{equation}

Moreover, there exists $u_{0}\in C^{\infty}(M\setminus\{x\})$ that realizes the infimum.
\end{thm}
\begin{proof}
We have seen that the infimum is finite. Now we show the existence of the minimizer $u_{0}$ by using direct methods in the calculus of variations by the following two steps.

{\em Step 1}. Let
\begin{equation}
m=\inf\left\{W(g, u, \tau)\mid u\in H^{1}(M), u>0, \left\|\frac{1}{(4\pi\tau)^{\frac{n}{4}}}u\right\|_{L^{2}(M)}=1\right\}>-\infty,
\end{equation}
and $\{u_{i}\}^{\infty}_{i=1}$ be a minimizing sequence, i.e.
\begin{equation}
u_{i}>0, \ \  \left\|\frac{1}{(4\pi\tau)^{\frac{n}{4}}}u_{i}\right\|_{L^{2}(M)}=1, \ \ \textit{for all} \ \ i,
\end{equation}
and
\begin{equation}
\lim_{i\rightarrow\infty}W(g, u_{i}, \tau)=m.
\end{equation}

By the work in \cite{DW17}, there exists constants $A=A(g)$, $C_{1}=C_{1}(g, A)$, and $C_{2}=C_{2}(g, A)$, such that for any $u\in H^{1}(M)$
\begin{equation}\label{WeightedSobolevNormEquivalence}
C_{1}\|u\|_{H^{1}(M)}\leq\int_{M}((R_{g}+A)u^{2}+4|\nabla u|^{2})d\vol_{g}\leq C_{2}\|u\|_{H^{1}(M)}.
\end{equation}
Here, the left inequality follows from Theorem 5.1 in \cite{DW17}, and the right inequality follows from the definition of the weighted Sobolev norm $\|\cdot\|_{H^{1}(M)}$ and the fact that $M$ and the cross section $N$ are compact.

Then by $(\ref{W-functionalLowerBound})$ and $(\ref{WeightedSobolevNormEquivalence})$, there exists a constant $B$ such that
\begin{equation}
\|u_{i}\|_{H^{1}(M)}\leq B,
\end{equation}
for all $i$. Thus, by Theorem 3.1 in \cite{DW17}, there exists a subsequence of the minimizing sequence $\{u_{i}\}$, which is still denoted by $\{u_{i}\}$, weakly converges to $u_{0}$ in $H^{1}(M)$, and strongly converges to $u_{0}$ in $L^{2}(M)$ for some $u_{0}\in H^{1}(M)$. And consequently, $u_{0}\geq 0$ a.e., and $\left\|\frac{1}{(4\pi\tau)^{\frac{n}{4}}}u_{0}\right\|_{L^{2}(M)}=1$.

{\em Step 2}. Now we will show that $W(g, u_{0}, \tau)\leq\lim\limits_{i\rightarrow\infty}W(g, u_{i}, \tau)=m$, and then $u_{0}$ is a minimizer.

For any $u, v\in H^{1}(M)$, let
\begin{equation}
(u, v)_{A}\equiv\int_{M}((R_{g}+A)u\cdot v+4\langle\nabla u, \nabla v\rangle)d\vol_{g}.
\end{equation}
Then by $(\ref{WeightedSobolevNormEquivalence})$, $(u, v)_{A}$ is an inner product on $H^{1}(M)$, and it induces a norm $\|\cdot\|_{A}$ that is equivalent to $H^{1}(M)$ norm. And we have
\begin{align*}
\|u_{i}\|^{2}_{A}
&=\|u_{0}\|^{2}_{A}+2(u_{0}, u_{i}-u_{0})_{A}+\|u_{i}-u_{0}\|^{2}_{A}\\
&\geq \|u_{0}\|^{2}_{A}+2(u_{0}, u_{i}-u_{0})_{A}.
\end{align*}
Becasue $u_{0}\in H^{1}(M)$ and $u_{i}$ weakly converges to $u_{0}$ in $H^{1}(M)$, we have
\begin{equation}
\lim_{i\rightarrow\infty}(u_{0}, u_{i}-u_{0})_{A}=0.
\end{equation}
Thus,
\begin{equation}
\lim_{i\rightarrow\infty}\|u_{i}\|^{2}_{A}\geq\|u_{0}\|^{2}_{A}.
\end{equation}
Then by $\lim_{i\rightarrow\infty}\|u_{i}\|_{L^{2}(M)}=\|u_{0}\|_{L^{2}(M)}$, we obtain
\begin{equation*}
\lim_{i\rightarrow\infty}\frac{\tau}{(4\pi\tau)^{\frac{n}{2}}}\int_{M}[\tau(R_{g}u^{2}_{i}+4|\nabla u_{i}|^{2})-nu^{2}_{i}]d\vol_{g}\geq\frac{\tau}{(4\pi\tau)^{\frac{n}{2}}}\int_{M}[\tau(R_{g}u^{2}_{0}+4|\nabla u_{0}|^{2})-nu^{2}_{0}]d\vol_{g}.
\end{equation*}

So it suffices to show that $\int_{M}u^{2}_{i}\ln u_{i}d\vol_{g}\rightarrow\int_{M}u^{2}_{0}\ln u_{0}d\vol_{g}$ as $i\rightarrow\infty$ for a subequence of the minimizing sequence $\{u_{i}\}$. As in the proof of Proposition 11.10 in \cite{AH10}, $\nabla(u^{2}\ln u)=(2u\ln u+u)\nabla u$, and for any $\gamma>0$ there exists constants $a, b>0$ such that $|u\ln u|\leq a+bu^{1+\gamma}$. Then for sufficiently small $\gamma>0$, we have
\begin{align*}
&\ \ \ \ \int_{M}|\nabla(u^{2}_{i}\ln u_{i})|d\vol_{g}\\
&\leq\int_{M}|u_{i}+2u_{i}\ln u_{i}|\cdot|\nabla u_{i}|d\vol_{g}\\
&\leq \left(\int_{M}|2a+u_{i}|^{2}d\vol_{g}\right)^{\frac{1}{2}}\left(\int_{M}|\nabla u_{i}|^{2}\right)^{\frac{1}{2}}+2b\left(\int_{M}|u_{i}|^{2+2\gamma}\right)^{\frac{1}{2}}\left(\int_{M}|\nabla u_{i}|^{2}d\vol_{g}\right)^{\frac{1}{2}}\\
&\leq C_{3},
\end{align*}
for a constant $C_{3}$ independent of $i$. Here, we use the Sobolev embedding $H^{1}(M)\subset W^{1, 2}(M)\subset L^{q}(M)$ for $1\leq q\leq\frac{2n}{n-2}$.

And
\begin{align*}
\int_{M}\chi|u^{2}_{i}\ln u_{i}|d\vol_{g}
&\leq\int_{M}\chi|u_{i}|\cdot|a+bu^{1+\gamma}_{i}|d\vol_{g}\\
&\leq\left(\int_{M}\chi^{2}|u_{i}|^{2}d\vol_{g}\right)^{\frac{1}{2}}\left(\int_{M}|a+bu_{i}|^{2}d\vol_{g}\right)^{\frac{1}{2}}\\
&\leq\|u_{i}\|_{H^{1}(M)}\left(a(\Vol_{g}(M))^{\frac{1}{2}}+b\left(\int_{M}|u_{i}|^{2+2\gamma}d\vol_{g}\right)^{\frac{1}{2}}\right)\\
&\leq C_{4},
\end{align*}
for a constant $C_{4}$ independent of $i$.

Thus,
\begin{equation}
\|u^{2}_{i}\ln u_{i}\|_{W^{1, 1}_{1-n}}\leq C_{3}+C_{4},
\end{equation}
for all $i$. Then by Proposition $\ref{CompactWeightedSobolevEmbedding}$, passsing to a subsequence we have $\lim\limits_{i\rightarrow\infty}\int_{M}u^{2}_{i}\ln u_{i}d\vol_{g}=\int_{M}u^{2}_{0}\ln u_{0}d\vol_{g}$.

Now we have obtained a minimizer $u_{0}\in H^{1}(M)$, and $u_{0}$ is a weak solution of the elliptic equation
\begin{equation}
-\Delta u+\frac{1}{4}R_{g}u-\frac{2}{\tau}u\ln u-\frac{n}{\tau}u-\frac{m}{\tau}u=0,
\end{equation}
where $m$ is the infimum  of the $W$-functional. The regularity of $u_{0}$ and $u_{0}>0$ can be shown locally. Thus the proof is the same as the compact smooth case, for details, see, e.g. p. 179 in \cite{AH10}.
\end{proof}
\end{section}


\begin{section}{Asymptotic behavior of the minimizer}
\noindent In this section, we obtain an asymptotic order for the minimizer near the singularity by using a weighted elliptic bootstrapping. For this, we need to establish some weighted Sobolev inequalities and weighted elliptic estimates. We first work on a finite asymptotic cone $(C_{\epsilon}(N)=(0,\epsilon)\times N, g=dr^{2}+r^{2}h_{r})$ satisfying the condition $AC_{1}$ defined in $(\ref{AC})$. In the following, we set
\begin{equation}
L\equiv-\Delta+\frac{1}{4}R.
\end{equation}

As in \S 8 in \cite{DW17}, we define weighted uniform $C^{k}_{\delta}$-norms on a finite cone $C_{\epsilon}(N)$ as
\begin{equation}\label{WeightedUniformNorm}
\|u\|_{C^{k}_{\delta}(C_{\epsilon(N)})}=\sup_{C_{\epsilon}(N)}\left(\sum^{k}_{i=0}r^{i-\delta}|\nabla^{i}u|\right),
\end{equation}
for $k\in\mathbb{N}$ and $\delta\in\mathbb{R}$. When $k=0$, we use $C_{\delta}$ to denote $C^{0}_{\delta}$.

Recall the homogeneity property of the weighted norms under scaling along the radial direction of the cone. Let $u(r, \theta)\in C^{\infty}_{0}(C_{\epsilon}(N))$, where $\theta$ is a local coordinate on $N$, and set
\begin{equation}
u_{a}(r, \theta)=u(ar, \theta),
\end{equation}
for a positive constant $a$. And let $C_{r_{1}, r_{2}}=(r_{1}, r_{2})\times N$ be an annulus on the finite metric cone $C_{\epsilon}(N)$, for $0\leq r_{1}<r_{2}\leq\epsilon$. Then by a simple change of variable, we can see
\begin{equation}\label{CNormHomogeneity}
\|u\|_{C^{k}_{\delta}(C_{ar_{1},ar_{2}})}=a^{-\delta}\|u_{a}\|_{C^{k}_{\delta}(C_{r_{1},r_{2}})}.
\end{equation}

Then similar as Theorem 1.2 in \cite{Bar86}, and the same as Lemma 8.1 in \cite{DW17}, by using $(\ref{SobolevHomogeneity})$, $(\ref{CNormHomogeneity})$, and the scaling technique, we obtain the following Sobolev inequality as an extension of Lemma 8.1 in \cite{DW17}.

\begin{lem}
If $\epsilon>0$ is sufficiently small, for any $u\in W^{2, p}_{\delta}(C_{\epsilon}(N))$ with $2>\frac{n}{p}+l$, we have
\begin{equation}
\|u\|_{C^{l}_{\delta}(C_{\epsilon}(N))}\leq C\|u\|_{W^{2, p}_{\delta}(C_{\epsilon}(N))},
\end{equation}
for a constant $C=C(g, n,k,\delta,\epsilon)$.

Moreover,
$$|\nabla^{l}u(r,x)|=o(r^{-l+\delta}) \ \ \text{as} \ \ r\rightarrow0.$$
\end{lem}

\begin{lem}
If $\epsilon$ is sufficiently small, $u\in W^{0, p}_{\delta}(C_{\epsilon}(N))$, and $Lu\in W^{0, p}_{\delta-2}(C_{\epsilon}(N))$, then
\begin{equation*}
\|u\|_{W^{2, p}_{\delta}(C_{\epsilon}(N))}\leq C\left(\|Lu\|_{W^{0, p}_{\delta-2}(C_{\epsilon}(N))}+\|u\|_{W^{0, p}_{\delta}(C_{\epsilon}(N))}\right),
\end{equation*}
for a constant $C=C(g, n,\delta,\epsilon)$.
\end{lem}

Similar as in $(\ref{WeightedUniformNorm})$, we also define weighted uniform $C^{k}_{\delta}$-norms on a compact Riemannian manifold $(M^{n}, g, x)$ with a single conical singularity at $x$ as
\begin{equation}
\|u\|_{C^{k}_{\delta}(M)}=\sup_{M\setminus\{x\}}\left(\sum^{k}_{i=0}\chi^{\delta-i}|\nabla^{i} u|\right),
\end{equation}
for any $u\in C^{k}(M\setminus\{x\})$, where $\chi$ is a weight function as in $(\ref{WeightFunction})$.

Then by using the usual Sobolev inequality and elliptic estimates on the interior part of compact manifolds with conical singularities, we directly obtain
\begin{lem}\label{SobolevInequality}
Let $(M^{n}, g, x)$ be a compact Riemannian manifold with a single conical singularity at $x$ satisfying the condition $AC_{1}$ defined in $(\ref{AC})$. Then for any $u\in W^{2, p}_{\delta}(M)$ with $2>\frac{n}{p}+l$, we have
\begin{equation}
\|u\|_{C^{l}_{\delta}(M)}\leq C\|u\|_{W^{2, p}_{\delta}(M)},
\end{equation}
for a constant $C=C(g, n, k, \delta)$.

Moreover,
$$|\nabla^{l}u(r,x)|=o(r^{-l+\delta}) \ \ \text{as} \ \ r\rightarrow0.$$
\end{lem}

\begin{lem}\label{EllipticEstimate}
Let $(M^{n}, g, x)$ be a compact Riemannian manifold with a single conical singularity at $x$ satisfying the condition $AC_{1}$ defined in $(\ref{AC})$. If $u\in W^{0, p}_{\delta}(M)$, and $Lu\in W^{0, p}_{\delta-2}(M)$, then
\begin{equation}
\|u\|_{W^{2, p}_{\delta}(M)}\leq C\left(\|Lu\|_{W^{0, p}_{\delta-2}(M)}+\|u\|_{W^{0, p}_{\delta}(M)}\right),
\end{equation}
for a constant $C=C(g, n, k, \delta)$.
\end{lem}

These weighted Sobolev inequalities and weighted elliptic estimates imply the following asymptotic order estimate for the minimizer of the $W$-functional near the conical singularities.

\begin{thm}\label{AsymptoticOrderOnCone}
Let $u$ be the minimizer of $W$-functional obtained in Theorem $\ref{ExistenceOfMinimizer}$. If the manifold satisfying the condition $AC_{1}$ defined in $(\ref{AC})$, then we have
$$u=o(r^{-\alpha}), \ \ \ \text{as} \ \ r\rightarrow0,$$
for any $\alpha>\frac{n}{2}-1$.
\end{thm}
\begin{proof}
Since $u$ satisfies the second order elliptic equation
\begin{equation}
Lu=\frac{2}{\tau}u\ln u+\frac{n+m}{\tau}u,
\end{equation}
and $u\in W^{1, 2}_{1-\frac{n}{2}}(M)$, where $m$ is the infimum of the $W$-functional,  by the weighted Sobolev embedding in Proposition $\ref{WeightedSobolevEmbedding}$, we have $u\in W^{0, p}_{1-\frac{n}{2}}(M)$, for any $1\leq p\leq\frac{2n}{n-2}$.

Because for each $\gamma>0$ there exists a constant $a(\gamma)$ such that $|u\ln u|\leq a(\gamma)+|u|^{1+\gamma}$, we have $u\ln u\in W^{0, p}_{(1-\frac{n}{2})(1+\gamma)}(M)\subset W^{0, p}_{(1-\frac{n}{2})(1+\gamma)-2}(M)$ for any $1\leq p\leq\frac{2n}{(n-2)}\frac{1}{(1+\gamma)}$ and any $\gamma>0$. So we have $Lu\in W^{0, p}_{(1-\frac{n}{2})(1+\gamma)-2}(M)$, since $u\in W^{0, p}_{1-\frac{n}{2}}(M)\subset W^{0, p}_{(1-\frac{n}{2})(1+\gamma)-2}(M)$.

Thus, by Lemma $\ref{EllipticEstimate}$, $u\in W^{2, p}_{(1-\frac{n}{2})(1+\gamma)}(M)$ for any $1\leq p\leq\frac{2n}{(n-2)}\frac{1}{(1+\gamma)}$ and any $\gamma>0$. If $2<n<6$, then by Lemma $\ref{SobolevInequality}$ we have obtained that $u=o(r^{-\alpha})$ as $r\rightarrow0$ for any $\alpha>\frac{n}{2}-1$, since $\gamma>0$ could be arbitrarily small.

If $n\geq6$, then using Proposition $\ref{WeightedSobolevEmbedding}$ again, we have $u\in W^{0, p}_{(1-\frac{n}{2})(1+\gamma)}(M)$ for any $1\leq p\leq\frac{2n}{(n-2)(1+\gamma)-4}$ and any $\gamma>0$, and $u\ln u\in W^{0, p}_{(1-\frac{n}{2})(1+\gamma)^{2}}(M)$ for any $1\leq p\leq\frac{2n}{[(n-2)(1+\gamma)-4]}\frac{1}{(1+\gamma)}$ and any $\gamma>0$. Then as before we have $Lu\in W^{0, p}_{(1-\frac{n}{2})(1+\gamma)^{2}-2}(M)$, and by Lemma $\ref{EllipticEstimate}$, we have $u\in W^{2, p}_{(1-\frac{n}{2})(1+\gamma)^{2}}(M)$, for any $1\leq p\leq\frac{2n}{[(n-2)(1+\gamma)-4]}\frac{1}{(1+\gamma)}$ and any $\gamma>0$.

For $n=6$, we can choose $p\geq1$, such that $2>\frac{n}{p}=\frac{6}{p}>2\gamma(1+\gamma)$. And then by Lemma $\ref{SobolevInequality}$ we have obtained that $u=o(r^{-\alpha})$ as $r\rightarrow0$ for any $\alpha>\frac{n}{2}-1$, since $\gamma>0$ could be arbitrarily small.

For $6<n<10$, because $\frac{2n}{[(n-2)(1+\gamma)-4]}\frac{1}{(1+\gamma)}<\frac{2n}{(n-6)(1+\gamma)^{2}}$, we can choose $p\geq1$ such that $2>\frac{n}{p}>\frac{(n-6)(1+\gamma)^{2}}{2}$ for sufficiently small $\gamma>0$. Thus $u=o(r^{-\alpha})$ as $r\rightarrow0$ for any $\alpha>\frac{n}{2}-1$, since $\gamma>0$ could be arbitrarily small.

Then for each fixed $n\geq10$, by repeating this process finitely many times, we can always obtain  that  $u=o(r^{-\alpha})$ as $r\rightarrow0$ for any $\alpha>\frac{n}{2}-1$.
\end{proof}
\end{section}



\end{document}